\documentclass[12pt]{amsart}

\allowdisplaybreaks

\usepackage{graphicx}

\usepackage{amssymb}
\usepackage{mathscinet}

\newtheorem{theorem}{Theorem}
\newtheorem{lemma}{Lemma}

\newtheorem{corollary}{Corollary}
\newtheorem*{theorem*}{Theorem}
\newtheorem*{proposition*}{Proposition}
\newtheorem*{metathm}{Meta-Theorem}
\theoremstyle{definition}
\newtheorem{definition}{Definition}
\theoremstyle{remark}
\newtheorem{remark}{Remark}

\usepackage[none]{hyphenat}
\newcommand{\norm}[1]{\left\Vert#1\right\Vert}

\newcommand{\set}[1]{\left\{#1\right\}}
\newcommand{\ip}[1]{\left\langle#1\right\rangle}

\begin{document}





\title[Boundary Representations of Positive Matrices]{A Characterization of Boundary Representations of Positive Matrices in the Hardy Space via the Abel Product}
\author{John E. Herr}
\address[John E. Herr]{Department of Mathematics, Butler University, Indianapolis, Indiana 46208}
\email{jeherr@butler.edu}
\author{Palle E. T. Jorgensen}
\address[Palle E.T. Jorgensen]{Department of Mathematics, University of Iowa, Iowa City, IA 52242}
\email{palle-jorgensen@uiowa.edu}
\author{Eric S. Weber}
\address[Eric S. Weber]{Department of Mathematics, Iowa State University, 396 Carver Hall, Ames, IA 50011}
\email{esweber@iastate.edu}
\date{\today}
\begin{abstract}
Spectral measures give rise to a natural harmonic analysis on the unit disc via a boundary representation of a positive matrix arising from a spectrum of the measure.  We consider in this paper the reverse: for a positive matrix in the Hardy space of the unit disc we consider which measures, if any, yield a boundary representation of the positive matrix.  We prove a characterization of those representing measures via a matrix identity by introducing a new operator product called the Abel Product.
\end{abstract}
\maketitle

\section{Introduction}

\subsection{The Szeg\H{o} Kernel}  
The classical Hardy space $H^2(\mathbb{D})$ consists of those holomorphic functions $f$ defined on $\mathbb{D}$ satisfying
\begin{equation}\label{Hardy}\lVert f\rVert^2_{H^2}:=\sup_{0<r<1}\int_{0}^{1}\lvert f(re^{2\pi ix})\rvert^2\,dx<\infty.\end{equation}
In addition, for each $f\in H^2(\mathbb{D})$, there exists a (unique) function $f^\ast\in L^2(\mathbb{T})$, which we shall call the Lebesgue boundary function of $f$, such that
\begin{equation}\label{LebBoundary}\lim_{r\rightarrow1^{-}}\int_{0}^{1}\lvert f(re^{2\pi ix})-f^\ast(e^{2\pi ix})\rvert^2\,dx=0.\end{equation}
Because the point-evaluation functionals on the Hardy space are bounded, the Hardy space is a reproducing kernel Hilbert space. Its kernel is the classical Szeg\H{o} kernel $k(w,z)=:k_w$, defined by
$$k_w(z):=\frac{1}{1-\overline{w}z}.$$
Because the mapping $f \mapsto f^{\ast}$ is an isometry, we have
\begin{equation}\label{szegorep}k(w,z):=\int_{0}^{1}k_w^\ast(e^{2\pi ix})\overline{k_z^\ast(e^{2\pi ix})}\,dx.\end{equation} 
Equation $\eqref{szegorep}$ shows that the Szeg\H{o} kernel reproduces itself with respect to what is, by some definition, its boundary. The measure on the circle used to define $k_z^\ast$ in \eqref{LebBoundary} is Lebesgue measure, as is the measure in $\eqref{szegorep}$. The intent of this paper is to show that among the functions in the Hardy space, there are a host of other kernels that reproduce with respect to their boundaries. However, these boundary functions will not be taken with respect to Lebesgue measure, but with respect to some other  measure on the circle $\mathbb{T}$, and the integration of these boundary functions will also be done with respect to this measure. Of interest are two main questions: which positive matrices does the Hardy space contain that reproduce themselves by boundary functions with respect to a given measure, and with respect to which measures will a positive matrix reproduce itself by boundary functions?  In \cite{HJW16a}, we focused on the first question whereas in the present paper we focus on the second.

\subsection{Boundary Representations}  The boundary behavior of functions in the Hardy space arises in a number of contexts, such as the spectral theory of the shift operator \cite{Clark72}, de Branges-Rovnyak spaces \cite{dBR66a,Sar94}, and ``pseudo-continuable'' functions \cite{Aleks89a,Pol93}

The boundary behavior of kernels in reproducing kernel Hilbert spaces has been considered by others, see for example \cite{KN77a} in the study of the Markov moment problem.  In the series \cite{AD84a,AD85a,ADLW06a}, much of the theory is developed under the a priori assumption that a reproducing kernel Hilbert space ``sits isometrically'' within an $L^2$ space, by which the authors mean that the elements of the Hilbert space possesses an $L^2$ boundary as we define below and the two norms coincide.  However, the authors also begin with a measure on the boundary and define a reproducing kernel by integrating the Cauchy kernel against the measure \cite[Lemma 6.4]{AD84a}.  In the previous paper \cite{HJW16a}, we extend this result to construct many kernels which reproduce themselves with respect to a fixed singular measure using the Kaczmarz algorithm \cite{KwMy01,HW17a}.  Other boundary representations were introduced in \cite{DJ11} to understand the nature of spectral measures.  In \cite{DJ11}, any spectral measure gives rise to a positive matrix on $\mathbb{D}$ which reproduces itself on the boundary with respect to the spectral measure.  As demonstrated in \cite{HJW16a}, however, the measure need not be spectral.

\begin{remark}In this paper, we will be dealing with measures $\mu$ on the unit circle. The unit circle $\mathbb{T}:=\{z\in\mathbb{C}:\lvert z\rvert=1\}$ and its topology shall be identified with $[0,1)$ via the relation $\xi=e^{2\pi ix}$ for $\xi\in\mathbb{T}$ and $x\in[0,1)$. We will regard the measures $\mu$ as being supported on $[0,1)$. A function $f(\xi)$ defined on $\mathbb{T}$ (for example, a boundary function) may be regarded as being in $L^2(\mu)$ if $f(e^{2\pi ix})\in L^2(\mu)$. For aesthetics, the inner product (norm) in $L^2(\mu)$ will be denoted $\langle\cdot,\cdot\rangle_{\mu}$ ($\lVert\cdot\rVert_{\mu}$) rather than $\langle\cdot,\cdot\rangle_{L^2(\mu)}$ ($\lVert\cdot\rVert_{L^2(\mu)}$). The subscript will be suppressed where context suffices.  A measure $\mu$ will be called \emph{singular} if it is a Borel measure that is singular with respect to Lebesgue measure.
\end{remark}

\begin{definition} \label{D:boundary} If $\mu$ is a finite Borel measure on $[0,1)$ and $f(z)$ is an analytic function on $\mathbb{D}$, we say that $f^{\star}\in L^2(\mu)$ is a (weak) $L^2(\mu)$-boundary function of $f$ if
\begin{equation*}\lim_{r\rightarrow1^-} f(re^{2\pi ix}) = f^{\star}(x)\end{equation*}
in the weak topology of $L^2(\mu)$.  
\end{definition}

\begin{remark}
In the pseudo-continuable functions literature, as well as the boundary representation literature, the limit is required to exist in norm.  For our purposes, the weak limit is more natural--we shall drop the descriptive ``weak'', though for the remainder of the paper, in all cases the boundary is a weak boundary as in Definition \ref{D:boundary}.  If a function possesses an $L^2(\mu)$-boundary function, then clearly that boundary function is unique as an element of $L^2(\mu)$. The $L^2(\mu)$-boundary function of a function $f:\mathbb{D}\rightarrow\mathbb{C}$ shall be denoted $f^{\star}_\mu$, but we omit the subscript when context precludes ambiguity.
\end{remark}

\begin{definition}A positive matrix (in the sense of E.~H.~Moore) on a domain $E$ is a function $K(z,w):E\times E\rightarrow\mathbb{C}$ such that for all finite sequences $\zeta_1,\zeta_2,\ldots,\zeta_n\in E$, the matrix
$$(K(\zeta_j,\zeta_i))_{ij}$$
is positive semidefinite. 
\end{definition}

Our interest is in positive matrices on $E=\mathbb{D}$, and more specifically those residing in $H^2(\mathbb{D})$.  We therefore desire to find subspaces of the Hardy space that not only are Hilbert spaces with respect to the $L^2(\mu)$-boundary norm, but are in fact reproducing kernel Hilbert spaces with respect to this norm. The classical Moore-Aronszajn Theorem connects positive matrices to reproducing kernel Hilbert spaces \cite{Aron50a} (see also \cite{PR16a}).

Let us then define two sets of interest:

\begin{definition}Let $\mu$ be a Borel measure on $[0,1)$. We define $\mathcal{K}(\mu)$ to be the set of positive matrices $K$ on $\mathbb{D}$ such that for each fixed $z\in\mathbb{D}$, $K(w,\cdot)$ possesses an $L^2(\mu)$-boundary $K^{\star}(w,\cdot)$, and $K(w,z)$ reproduces itself with respect to integration of these $L^2(\mu)$-boundaries, i.e.
\begin{equation} \label{Eq:reproduce}
K(w,z)=\int_{0}^{1}K^{\star}(w,x)\overline{K^{\star}(z,x)}\,d\mu(x)
\end{equation}
for all $z,w\in\mathbb{D}$.\end{definition}

\begin{definition}Let $K$ be a positive matrix on $\mathbb{D}$. We define $\mathcal{M}(K)$ to be the set of nonnegative Borel measures $\mu$ on $[0,1)$ such that for each fixed $z\in\mathbb{D}$, $K(w, \cdot)$ possesses an $L^2(\mu)$-boundary $K^{\star}(w, \cdot)$, and $K(w,z)$ reproduces itself with respect to integration of these $L^2(\mu)$-boundaries as in Equation (\ref{Eq:reproduce}).\end{definition}

\subsection{Kernels from a Coefficient Matrix}

Let $C = (c_{mn})_{mn}$ be an infinite matrix, where $m,n \geq 0$.  We consider the formal power series
\begin{equation} \label{Eq:KC}
K_{C}(w, z) = \sum_{n=0}^{\infty} \sum_{m=0}^{\infty} c_{mn} \overline{w}^{m} z^{n}.
\end{equation}
If the matrix $C$ defines a positive bounded linear operator on $\ell^{2}(\mathbb{N}_{0})$, then $K_{C}$ is a positive matrix in the sense of E.H. Moore, and thus defines a reproducing kernel Hilbert space.  Reproducing kernels of this form were considered in detail in  \cite{AMP92a,AMP94a}.  Other reproducing kernel Hilbert spaces in the Hardy space were considered by de Branges and Rovnyak in \cite{dBR66a} (see also \cite{Sar94}) leading to de Branges' solution of the Bieberbach conjecture \cite{dB85a}.  They were also considered in \cite{LP11a} in an attempt to solve the Kadison-Singer problem/Feichtinger conjecture, which has now been resolved \cite{MSS15a}.

For a given $C$ which defines a positive matrix as in Equation (\ref{Eq:KC}), we wish to determine which Borel measures on $\mathbb{T}$, if any, are in $\mathcal{M}(K_{C})$.  We shall approach the question via the following heuristic:

\begin{metathm}
A measure $\mu$ is in $\mathcal{M}(K_{C})$ if and only if the matrix equation $C = CMC$ is satisfied, where the matrix $M = ( \hat{\mu}(n-m) )_{mn}$.
\end{metathm}

We describe this as a meta-theorem for the following reason: since $M$ as a matrix is in general not a bounded operator on $\ell^2(\mathbb{N}_{0})$, the expression $CMC$ may not be well-defined.  We established in \cite{HJW17b} the meta-theorem for the two special cases: i) for diagonal matrices $C$, and ii) for $C$ and $\mu$ for which $M$ are bounded operators on $\ell^{2}(\mathbb{N}_{0})$.  (A description of those $\mu$ that have the property that $M$ is bounded is found in \cite{Cas00a,DHSW11}, see also \cite{Lai12,DL14a}.)  In both cases, the matrix product $CMC$ exists.  We consider in the present paper the possibility that $CMC$ does not exist.  We do so by introducing a new matrix operation we call the Abel product, which allows in some cases a definition for a matrix product when the ordinary matrix product does not exist, in analogy to Abel summation allows the definition of the sum of a series which may not converge in the ordinary sense.

Let $C$ be a bounded, positive operator from $\ell^2$ to itself, represented as an infinite matrix by $C=(c_{mn})_{mn}$. Define $K_C(w,z):\mathbb{D}\times\mathbb{D}\to\mathbb{C}$ as in Equation (\ref{Eq:KC});  we can for convenience write 
\begin{align*}
K_C(w,z) &=\sum_{n=0}^{\infty}\sum_{m=0}^{\infty}c_{mn}\overline{w}^mz^n \\
&=\sum_{m=0}^{\infty}\sum_{n=0}^{\infty}c_{mn}\overline{w}^mz^n\\
&=\ip{\left(\sum_{n=0}^{\infty}c_{mn}z^n\right)_{m},\vec{w}}_{\ell^2}\\
&=\ip{C \vec{z}, \vec{w} }_{\ell^2}
\end{align*}
where for $z \in \mathbb{D}$, $\vec{z} := (z^{n})_{n} \in \ell^{2}$.  Note that we used absolute convergence in the second line above.

Given a sequence of vectors $\set{x_n}_{n=0}^{\infty}$ in a Hilbert space $\mathbb{H}$, we define the synthesis operator of $\set{x_n}$, $S_{x}:\mathbb{C}^{\mathbb{N}}\to\mathbb{H}$ by
$$S_{x} [(c_n)_{n}] =\sum_{n=0}^{\infty}c_nx_n$$
where the convergence on the right is in the norm of $\mathbb{H}$ and the analysis operator of $\set{x_n}$, $A_{x}:\mathbb{H}\to\mathbb{C}^\mathbb{N}$ by 
$$A_{x}[f]= (\ip{f,x_n}_{\mathbb{H}})_{n=0}^{\infty}.$$
(It is understood that the domain of $S_x$ is not $\mathbb{C}^{\mathbb{N}}$ itself, but rather a subset thereof on which the series converges. Depending on the situation, $S_x$ and $A_x$ are understood to have smaller domains and codomains than those above.) 

Since $C$ is a positive matrix, it is the Gramian of some sequence of vectors. That is to say, there exists some sequence of vectors $\set{x_n}_{n=0}^{\infty}$ in some Hilbert space $\mathbb{H}$ such that $$C=\left(\ip{x_m,x_n}_{\mathbb{H}}\right)_{mn}.$$
Observe that 
$$A_xS_x [(c_{n})_{n}]=A_x\left(\sum_{n=0}^{\infty}c_nx_n\right)= \left( \sum_{n=0}^{\infty}c_n\ip{x_n,x_m}_{\mathbb{H}} \right)_{m}.$$
Thus, the composition $A_xS_x=\left(\ip{x_n,x_m}\right)_{mn}$ as an operator on sequences is the transpose of the Gramian of the $\set{x_n}$. It follows that $C$ can be realized by some sequence $\set{x_n}\subset\mathbb{H}$ as $C=(A_xS_x)^T.$

We denote by $S_{\overline{e}}$ and $A_{\overline{e}}$ the synthesis and analysis operators for $\{ \overline{e_{n}} \} \subset L^2(\mu)$, respectively.  Note that $\overline{e_{n}} = e_{-n}$.  It is easily seen that $(A_{\overline{e}} S_{\overline{e}})^{T} = A_{e} S_{e}$, and therefore the matrix $M = ( \hat{\mu}(n-m) )_{mn}$, which is the Grammian matrix of the $\{ e_{n} \}_{n=0}^{\infty} \subset L^2(\mu)$, can be factored as 
\begin{equation} \label{Eq:MAS}
M = A_{\overline{e}} S_{\overline{e}}
\end{equation}
as a matrix.  We can also formalize the factorization in Equation (\ref{Eq:MAS}) as follows.

\begin{lemma} \label{L:Mbounded}
The mappings 
\[ S_{e} : \ell^{1} \to L^2(\mu) : (c_n)_{n} \mapsto \sum_{n=0}^{\infty} c_{n} e_{n} \]
and 
\[ A_{e} : L^{2}(\mu) \to \ell^{\infty} : f \mapsto ( \langle f, e_{n} \rangle_{\mu} )_{n} \]
are bounded operators.  Likewise for $S_{\overline{e}}$ and $A_{\overline{e}}$.  Consequently, the matrix $M$ defines a bounded operator from $\ell^{1}$ to $\ell^{\infty}$ and $M = A_{\overline{e}} S_{\overline{e}}$.
\end{lemma}

\begin{proof}
The mapping $S_{e}$ is well-defined and bounded by absolute summability, while the mapping $A_{e}$ is well-defined and bounded by the Cauchy-Schwarz inequality.  It follows that the composition is bounded, and the matrix $M$ represents the composition $ A_{\overline{e}} S_{\overline{e}}$ as argued above.
\end{proof}

\begin{definition}By $D_s$ we shall mean the operator from $\ell^\infty$ to $\ell^1$ given by the diagonal matrix whose diagonal is the vector $\vec{s}$, where $0<s<1$.  Thus, $D_{s} [ (x_{n})_{n} ] = ( s^{n} x_{n} ) _{n}$. \end{definition}

\begin{definition}We shall define $$V:=\text{span}\set{\vec{v}:v\in\mathbb{D}}.$$ Note that $V$ is a proper, dense subspace of $\ell^2$.  \end{definition}

\begin{definition}Let $X$ and $Y$ be inner product spaces. Let $S$ be a (possibly unclosed) subspace of $\ell^1$. Let $T_1:X\to\ell^\infty$ and $T_2:S\to Y$ be (possibly unbounded) linear operators. Suppose that $D_sT_1X\subseteq S$ for all $0<s<1$. If there exists a bounded linear operator $A:X\to Y$ such that
\begin{equation} \label{Eq:limit}
\lim_{s\to1^-}\ip{T_2 D_sT_1x,y}_{Y}=\ip{Ax,y}_{Y}
\end{equation}
for all $x\in X$ and $y\in Y$, then we say that $A$ is the Abel product of $T_2$ and $T_1$, and we denote this product by $T_2\circledast T_1$.  If $X$ and/or $Y$ are subspaces of larger spaces, then the existence of the limit may depend on the $X$ and $Y$, and in fact may not exist for other subspaces.  We indicate this dependence as $T_2\circledast T_1 = A \restriction_{X,Y}$.
\end{definition}

\begin{remark} A few remarks regarding the Abel product:

\begin{enumerate}
\item The name ``Abel product`` is inspired by Abel summation:  if $T_{2}$ and $T_{1}$ are matrices whose Abel product exists on the finite span of the standard basis vectors $\{ \delta_{n} \}$ in $\ell^{2}$, then 
\[ \lim_{s \to 1^{-} } \ip{T_2 D_sT_1 \delta_{n} , \delta_{m} }_{\ell^2} = \lim_{s \to 1^{-}} \sum_{k = 0}^{\infty} (T_{2})_{m k} s^{k} (T_{1})_{k n} \]
which is the Abel sum of the ordinary matrix product of $T_{2}$ and $T_{1}$.
\item The Abel product extends the ordinary operator (matrix) product as follows: if $T_{2}$ and $T_{1}$ are bounded operators, then $T_{2} \circledast T_{1} = T_{2} \circ T_{1}$.
\item If $X$ and $Y$ are complete spaces, the existence of the limit in (\ref{Eq:limit}) for all $x\in X$ and $y\in Y$ is by the Uniform Boundedness Principle enough to imply the existence of $A$.
\item The same technique in \cite{AMP92a,AMP94a} for dealing with infinite matrices which are unbounded operators on $\ell^{2}(\mathbb{N}_{0})$: in those papers the authors ``pre-condition" an unbounded operator with the same diagonal matrix as here, but the authors in \cite{AMP92a,AMP94a} use the diagonal matrix to effectively perform a variable substitution, and do not consider the limit as we do here.
\end{enumerate}
\end{remark}



\section{A Matrix Characterization via the Abel Product}

Our main result will establish a  characterization of when $\mu$ is a representing measure for $K_{C}$.  Theorem \ref{Th:main} says that $\mu \in \mathcal{M}(K_{C})$ if and only if $ C = (C \circledast A_{\overline{e}}) (S_{\overline{e}} \circledast C)$.  If in the special case that $M$ is a bounded operator from $\ell^2$ to $\ell^2$, i.e. both $A_{\overline{e}}$ and $S_{\overline{e}}$ are bounded (which occurs when $\{ e_{n} \} \subset L^2(\mu)$ is a Bessel sequence,
then we have the following consequence of Theorem \ref{Th:main}:
\begin{align*}
C &= (C \circledast A_{\overline{e}}) (S_{\overline{e}} \circledast C) \\
&= (C \circ A_{\overline{e}}) (S_{\overline{e}} \circ C) \\
&= C (A_{\overline{e}} \circ S_{\overline{e}}) C \\
&= CMC.
\end{align*}
Therefore, the Abel product will allow us to rigorously extend this heuristic to the case when $M$ is not a bounded operator.

\begin{theorem}\label{firsttheorem}Suppose $C:\ell^2\to\ell^2$ is a bounded positive operator representable by an infinite matrix, and form the positive matrix $K_C(w,z):=\ip{C\vec{z},\vec{w}}.$ Suppose $K_C(w,z)$ has weak $L^2(\mu)$-boundaries in the sense that for each fixed $w\in\mathbb{D}$, there exists a function $K_w^\star\in L^2(\mu)$ such that for every $h(x)\in L^2(\mu)$,
$$\lim_{s\to1^{-}}\ip{K_C(w,se^{2\pi ix}),h(x)}_{\mu}=\ip{K_w^\star(x),h(x)}_{\mu}.$$
Suppose further that $K_C(w,z)$ reproduces itself with respect to these weak boundaries, i.e.
$$K_C(w,z)=\ip{K_w^\star,K_z^\star}_\mu.$$
Then there exists a bounded linear operator $L:\ell^2\to L^2(\mu)$ such that 
\begin{equation}\label{L-1}K_w^\star=L\vec{\overline{w}}\end{equation} for all $w\in\mathbb{D}$, and
\begin{equation}\label{L-2}\lim_{s\to1^{-}}\ip{S_eD_sC^Tv,h(x)}_\mu=\ip{Lv,h(x)}_\mu\end{equation}
for all $v\in V$ and $h\in L^2(\mu)$. Consequently, $L\restriction_{V,L^2(\mu)}=S_e\circledast C^T.$\end{theorem}

\begin{proof}
Observe that
\begin{align}
\ip{K^\star_w(x),h(x)}_\mu&=\lim_{s\to1^{-}}\ip{K(w,se^{2\pi ix}),h(x)}_{\mu} \notag \\
&=\lim_{s\to1^{-}}\ip{\sum_{n=0}^{\infty}\sum_{m=0}^{\infty}c_{mn}\overline{w}^ms^ne^{2\pi inx},h(x)}_{\mu} \notag \\
&=\lim_{s\to1^{-}}\ip{\sum_{n=0}^{\infty}\left(C^T\vec{\overline{w}}\right)_{n}s^ne^{2\pi inx},h(x)}_{\mu} \notag \\
&=\lim_{s\to1^-}\ip{\sum_{n=0}^{\infty}\left(D_sC^T\vec{\overline{w}}\right)_{n}e^{2\pi inx},h(x)}_{\mu} \notag \\
&=\lim_{s\to1^-}\ip{S_eD_sC^T\vec{\overline{w}},h(x)}_{\mu}.  \label{Eq:lim}
\end{align}
Define 
\begin{equation} \label{Eq:L}
L:V\to L^2(\mu) : (\sum_{j=0}^{N-1}\alpha_j\vec{w_j})=\sum_{j=0}^{N-1}\alpha_jK_{\overline{w_j}}^\star .
\end{equation}
We claim that $L$ is well-defined. Indeed, suppose that $\sum_{j=0}^{N-1}\alpha_j\vec{w_j}=\vec{0}$ for some distinct $\vec{w_j}$'s. Then in particular, the first $N$ entries of $\sum_{j=0}^{N-1}\alpha_j\vec{w_j}$ are equal to $0$. However, this is impossible, because the $N\times N$ Vandermonde matrix 
$$V=\left((\vec{w_m})_n\right)_{mn}$$
is nonsingular, since the $\vec{w_j}$ are distinct. Thus, it has linearly independent rows. By construction, then, $L$ is linear, and $\eqref{L-1}$ and therefore $\eqref{L-2}$ hold. 

We next claim that $L$ is bounded on $V$, and hence can be extended to all of $\ell^2$.  Because $K_C(w,z)$ reproduces itself with respect to its boundaries,
\begin{align*}\norm{L\left(\sum_{j=0}^{N}\alpha_j\vec{w_j}\right)}^2_\mu&=\norm{\sum_{j=0}^{N}\alpha_jK_{\overline{w_j}}^\star}_{\mu}^2\\&=\sum_{j=0}^{N}\sum_{k=0}^{N}\alpha_j\overline{\alpha_k}K(\overline{w_j},\overline{w_k})\\
&=\sum_{j=0}^{N}\sum_{k=0}^{N}\alpha_j\overline{\alpha_k}\ip{C\vec{\overline{w_k}},\vec{\overline{w_j}}}_{\ell^2}\\
&=\ip{C\left(\sum_{k=0}^{N}\overline{\alpha_k}\vec{\overline{w_k}}\right),\sum_{j=0}^{N}\overline{\alpha_j}\overline{w_j}}_{\ell^2}\\
&\leq\norm{C}\norm{\sum_{j=0}^{N}\overline{\alpha_j}\vec{\overline{w_j}}}^2_{\ell^2}\\
&=\norm{C}\norm{\sum_{j=0}^{N}\alpha_j\vec{w_j}}^2_{\ell^2}.\end{align*}
It now follows from Equations (\ref{Eq:lim}) and (\ref{Eq:L}) that $L\restriction_{V,L^2(\mu)}=S_e\circledast C^T$.
\end{proof}

\begin{lemma} \label{L:conjugate}
For $v \in \ell^2$, $s \in (0,1)$, we have
\begin{equation} \label{Eq:conjugate}
S_{\overline{e}} D_{s} C v = \overline{ S_{e} D_{s} C^{T} \overline{v} }.
\end{equation}
\end{lemma}
\begin{proof}
We calculate
\begin{align*}
S_{\overline{e}} D_{s} C v &= S_{\overline{e}} D_{s} \left( \sum_{n=0}^{\infty} c_{mn} v_{n} \right)_{m} \\
&= \sum_{m = 0}^{\infty} s^{m} \left( \sum_{n=0}^{\infty} c_{mn} v_{n} \right) \overline{e_{m}}
\end{align*}
whereas
\begin{align*}
S_{e} D_{s} C^{T} \overline{v} &= S_{e} D_{s} \left( \sum_{n=0}^{\infty} c_{nm} \overline{v_{n}} \right)_{m} \\
&= \sum_{m = 0}^{\infty} s^{m} \left( \sum_{n=0}^{\infty} c_{nm} \overline{v_{n}} \right) e_{m}.
\end{align*}
\end{proof}

\begin{corollary}\label{ConjAbelSum}  Under the hypotheses of Theorem \ref{firsttheorem}, there exists a bounded linear operator $\tilde{L}:\ell^2\to L^2(\mu)$ such that
$$\lim_{s\to1^-}\ip{S_{\overline{e}}D_sCv,h}_{\mu}=\ip{\tilde{L}v,h}_{\mu}$$ for all $v\in V$ and $h\in L^2(\mu)$. Consequently, $S_{\overline{e}}\circledast C=\tilde{L}\restriction_{V,L^2(\mu)}$.
\end{corollary}

\begin{proof}Define $\tilde{L}:\ell^2\to L^2(\mu)$ by $$\tilde{L}v:=\overline{L\overline{v}}.$$ It is easy to see that $\tilde{L}$ is linear and has the same bound as $L$. Let $h\in L^2(\mu)$. We have:
\begin{align*}\lim_{s\to 1^-}\ip{S_{\overline{e}}D_s Cv,h}_{\mu}&=\lim_{s\to1^-}\int_{0}^{1} [S_{\overline{e}}D_s Cv](x)\overline{h(x)}\,d\mu(x)\\
&=\lim_{s\to1^-}\int_{0}^{1}\overline{ [S_eD_sC^T\overline{v}](x)}\overline{h(x)}\,d\mu(x)\\
&=\int_{0}^{1} \left([\overline{L\overline{v}}](x) \right) \overline{h(x)}\,d\mu(x)\\
&=\ip{\tilde{L}v,h}_{\mu}.
\end{align*}

\end{proof}

\begin{lemma}\label{swappinglemma} Under the hypotheses of Theorem \ref{firsttheorem}, for all $w,z\in\mathbb{D}$, $$\ip{S_e\circledast C^T\vec{\overline{w}},S_e\circledast C^T\vec{\overline{z}}}=\ip{S_{\overline{e}}\circledast C\vec{z},S_{\overline{e}}\circledast C\vec{w}}.$$
\end{lemma}

\begin{proof} By Theorem \ref{firsttheorem} and Corollary \ref{ConjAbelSum}, both $S_{e} \circledast C^{T}$ and $S_{\overline{e}} \circledast C$ exist on $V$.  We have by Lemma \ref{L:conjugate}:
\begin{align*}\ip{S_e\circledast C^T\vec{\overline{w}},S_e\circledast C^T\vec{\overline{z}}}_{\mu}
&=\lim_{s\to 1^-}\lim_{r\to1^-}\int_{0}^{1} [S_e D_rC^T\vec{\overline{w}}](x)\overline{[S_eD_sC^T\vec{\overline{z}}](x)}\,d\mu(x)\\
&=\lim_{s\to 1^-}\lim_{r\to1^-}\int_{0}^{1}\overline{[S_{\overline{e}}D_rC\vec{w}](x)}[S_{\overline{e}}D_sC\vec{z}](x)\,d\mu(x)\\
&=\ip{S_{\overline{e}}\circledast C\vec{z},S_{\overline{e}}\circledast C\vec{w}}_{\mu}.
\end{align*}
\end{proof}

\begin{lemma}\label{adjointlemma}Suppose there exists a bounded operator $L:\ell^2\to L^2(\mu)$ such that $L\restriction_{V,L^2(\mu)}=S_e\circledast C^T$. Then for all $h\in L^2(\mu)$ and $v\in V$, $$\lim_{r\to1^-}\ip{C^TD_r A_eh,v}_{\ell^2}=\ip{L^\ast h,v}_{\ell^2}.$$ Consequently, $C^T\circledast A_e=L^\ast \restriction_{L^2(\mu),V}$.\end{lemma}
\begin{proof}
We calculate:
\begin{align*}\lim_{r\to1^-}\ip{C^TD_rA_eh,v}_{\ell^2}
&=\lim_{r\to1^-}\ip{ \left(\sum_{n=0}^{\infty}c_{nm}r^n\ip{h,e_n}_\mu\right)_m,v}_{\ell^2}\\
&=\lim_{r\to1^-}\sum_{m=0}^{\infty}\sum_{n=0}^{\infty}\overline{v_m}c_{nm}r^n\ip{h,e_n}_\mu\\
&=\lim_{r\to1^-}\sum_{n=0}^{\infty}\sum_{m=0}^{\infty}\overline{v_m}c_{nm}r^n\ip{h,e_n}_\mu &\text{  [by absolute summability]}\\
&=\lim_{r\to1^-}\ip{h,\sum_{n=0}^{\infty}\sum_{m=0}^{\infty}v_mc_{mn}r^ne_n}_\mu\\
&=\lim_{r\to1^-}\ip{h,S_eD_rC^Tv}_\mu\\
&=\ip{h,Lv}_{\mu}\\
&=\ip{L^\ast h,v}_{\ell^2}.
\end{align*}
\end{proof}

We are now in a position to prove our main result, which is a characterization of when a Borel measure $\mu$ is a representing measure for a positive matrix $K_{C}$ on $\mathbb{D}$.  Here we formalize the heuristic $C = CMC$ using the Abel product.

\begin{theorem} \label{Th:main}
Suppose $C$ is a positive bounded operator on $\ell^2$ and $\mu$ is a Borel measure on $[0,1)$.  Then $\mu\in\mathcal{M}(K_C)$ if and only if $C\restriction_{V,\ell^2}=(C\circledast A_{\overline{e}})(S_{\overline{e}}\circledast C)$.\end{theorem}

\begin{proof}If $\mu\in\mathcal{M}(K_C)$, then by definition, $K_C(w,z)$ has weak boundaries and reproduces with respect to those boundaries, and so Theorem \ref{firsttheorem} applies. Hence, for any $v=\sum_{j=0}^{M-1}\alpha_j\vec{z_j}\in V$ and $w=\sum_{k=0}^{M-1}\beta_k\vec{w_k}\in V$,
\begin{align*}\ip{C\left(\sum_{j=0}^{M-1}\alpha_j\vec{z_j}\right),\sum_{k=0}^{N-1}\beta_k\vec{w_k}}_{\ell^2}&=\sum_{j=0}^{M-1}\sum_{k=0}^{N-1}\alpha_j\overline{\beta_k}\ip{C\vec{z_j},\vec{w_k}}_{\ell^2}\\
&=\sum_{j=0}^{M-1}\sum_{k=0}^{N-1}\alpha_j\overline{\beta_k}K_C(w_k,z_j)\\
&=\sum_{j=0}^{M-1}\sum_{k=0}^{N-1}\alpha_j\overline{\beta_k}\ip{K_{w_k}^\star,K_{z_j}^\star}_\mu\\
&=\sum_{j=0}^{M-1}\sum_{k=0}^{N-1}\alpha_j\overline{\beta_k}\ip{L\vec{\overline{w_k}},L\vec{\overline{z_j}}}_\mu\\
&=\sum_{j=0}^{M-1}\sum_{k=0}^{N-1}\alpha_j\overline{\beta_k}\ip{S_e\circledast C^T\vec{\overline{w_k}},S_e\circledast C^T\vec{\overline{z_j}}}_\mu\\
&=\sum_{j=0}^{M-1}\sum_{k=0}^{N-1}\alpha_j\overline{\beta_k}\ip{S_{\overline{e}}\circledast C\vec{z_j},S_{\overline{e}}\circledast C\vec{w_k}}_\mu&\text{[by Lemma \ref{swappinglemma}]}\\
&=\sum_{j=0}^{M-1}\sum_{k=0}^{N-1}\alpha_j\overline{\beta_k}\ip{(C\circledast A_{\overline{e}})(S_{\overline{e}}\circledast C)\vec{z_j},\vec{w_k}}_{\ell^2}&\text{[by Lemma \ref{adjointlemma}]}\\
&=\ip{(C\circledast A_{\overline{e}})(S_{\overline{e}}\circledast C)\left(\sum_{j=0}^{M-1}\alpha_j\vec{z_j}\right),\sum_{k=0}^{N-1}\beta_k\vec{w_k}}_{\ell^2}.
\end{align*}
Since $V$ is dense in $\ell^2$, by continuity of the inner product the above holds not just for $w \in V$ but for all $w\in\ell^2$, and hence $C\restriction_{V,\ell^2}=(C\circledast A_{\overline{e}})(S_{\overline{e}}\circledast C)$.

Conversely, suppose $C\restriction_{V,\ell^2}=(C\circledast A_{\overline{e}})(S_{\overline{e}}\circledast C).$ Since $S_{\overline{e}}\circledast C$ is assumed to exist boundedly on $V$, there exists a bounded extension $\tilde{L}:\ell^2\to L^2(\mu)$ of $S_{\overline{e}}\circledast C$. Lemma \ref{adjointlemma} applies to show that $(C^T\circledast A_{\overline{e}})=\tilde{L}^\ast \restriction_{L^2(\mu),V}$, and $S_{e}\circledast C^T$ exists by the proof of Corollary \ref{ConjAbelSum}. Let $h(x)\in L^2(\mu)$. We have
\begin{align*}\lim_{s\to1^-}\ip{K_C(w,se^{2\pi ix}),h(x)}_{\mu}&=\lim_{s\to1^-}\ip{\sum_{n=0}^{\infty}\sum_{m=0}^{\infty}c_{mn}\overline{w}^ms^ne^{2\pi inx},h(x)}_{\mu}\\
&=\lim_{s\to1^-}\ip{S_{e}D_sC^T\vec{\overline{w}},h}_{\mu}\\
&=\ip{(S_e\circledast C^T)\vec{\overline{w}},h}_{\mu},
\end{align*}
which shows that $K_C(w,z)$ possesses weak $L^2(\mu)$ boundaries $K_w^\star=(S_e\circledast C^T)\vec{\overline{w}}$.  Then observe that by Lemma \ref{swappinglemma}
\begin{align*}
\ip{K_w^\star,K_z^\star}_{\mu}&=\ip{(S_e\circledast C^T)\vec{\overline{w}},(S_e\circledast C^T)\vec{\overline{z}}}_{\mu}\\
&=\ip{(S_{\overline{e}}\circledast C)\vec{z},(S_{\overline{e}}\circledast C)\vec{w}}_{\mu} &\text{[by Lemma \ref{swappinglemma}]}\\
&=\ip{\tilde{L}\vec{z},\tilde{L}\vec{w}}_{\mu} &\text{[by Corollary \ref{ConjAbelSum}]}\\
&=\ip{\tilde{L}^\ast\tilde{L}\vec{z},\vec{w}}_{\ell^2}\\
&=\ip{(C \circledast A_{\overline{e}})(S_{\overline{e}}\circledast C)\vec{z},\vec{w}}_{\ell^2}\\
&=\ip{C\vec{z},\vec{w}}_{\ell^2}&\text{[by assumption]}\\
&=K_C(w,z).
\end{align*}
Thus, $\mu\in\mathcal{M}(K_C)$.

\end{proof}

\section{Examples}

Here are a few examples to illustrate the Abel product.  First, consider the matrices
\[ B = \begin{pmatrix} 1 & 1 & 1 & \dots \end{pmatrix}, \qquad A = \begin{pmatrix} 1 & -1 & 1 & -1 &\dots \end{pmatrix}^{T}. \]
We view $A$ as being a bounded operator from $\mathbb{C} = \ell^{2}(\{0\})$ to $\ell^{\infty}(\mathbb{N}_{0})$, and $B$ a bounded operator from $\ell^{1}(\mathbb{N}_{0})$ to $\mathbb{C}$.  Note that the neither the matrix product nor the composition between $B$ and $A$ exist.  However, for $x,y \in \mathbb{C}$, we have
\[ \langle B D_{r} A x, y \rangle = \sum_{n=0}^{\infty} (-r)^{n} x \overline{y} \to \dfrac{1}{2} \langle x, y \rangle. \]
Thus, the Abel product exists, and $(B \circledast A) x = \dfrac{1}{2}x$.


We now turn to an example that illustrates Theorem \ref{Th:main} with a measure $\mu$ and a matrix $C$ for which the ordinary products $S_{\overline{e}}\circ C$ and $C\circ A_{\overline{e}}$ do not exist. Let $E\subset[0,1)$ be a finite set of points, and let $\mu$ be the normalized, equal-weight point-mass measure supported on $E$. By \cite[Theorem VIII.1.14]{Zygmund} (see also \cite{KK66a,PP15a}), there exists a continuous function $\psi$ in the disc algebra $A(\mathbb{D})$ whose Fourier series $\sum_{n=0}^{\infty}a_ne^{2\pi int}$ diverges on $E$. Set 
\[ \varphi(t) = \dfrac{ \psi( e^{2 \pi i t} )}{ \| \psi( e^{2 \pi i t} ) \|_{\mu} }. \]
Note that $\varphi$ is continuous, so that its Fourier series is Abel summable to itself pointwise everywhere, and that $\norm{\varphi}_\mu=1$. 

Now, define $\vec{x}=(\ip{e_n,\varphi}_2)_{n=0}^{\infty}$. That is, $\vec{x}$ is the sequence of conjugated Fourier coefficients of $\varphi$. Let $C=\vec{x}\otimes\vec{x}=(x_m\overline{x_n})_{mn}$. Since $\mu$ is a finite sum of point-masses, $L^2(\mu)$-norm convergence and pointwise convergence are equivalent. Thus, we compute that for any $\vec{y}\in\ell^2(\mathbb{N}_0)$ and $h\in L^2(\mu)$,
\begin{align*}\lim_{r\to1^-}\ip{S_{\overline{e}}D_rC\vec{y},h}_{\mu}&=\lim_{r\to1^-}\ip{S_{\overline{e}}D_r\left(x_m\sum_{n=0}^{\infty}\overline{x_n}y_n\right)_m,h}_{\mu}\\
&=\lim_{r\to1^-}\ip{S_{\overline{e}}\left(r^mx_m\ip{\vec{y},\vec{x}}_{\ell^2}\right)_m,h}_{\mu}\\
&=\lim_{r\to1^-}\ip{\sum_{m=0}^{\infty}r^mx_m\ip{\vec{y},\vec{x}}_{\ell^2}\overline{e_m},h}_{\mu}\\
&=\ip{\lim_{r\to1^-}\sum_{m=0}^{\infty}r^mx_m\ip{\vec{y},\vec{x}}_{\ell^2}\overline{e_m},h}_{\mu}\\
&=\ip{\ip{\vec{y},\vec{x}}_{\ell^2}\overline{\varphi},h}_{\mu}
\end{align*}
This shows that $S_{\overline{e}}\circledast C=\ip{\cdot,\vec{x}}_{\ell^2}\overline{\varphi}$ (formally in the weak operator topology, but since $L^2(\mu)$ is finite dimensional, we obtain convergence in the strong operator topology). We note that, by construction, for any $t \in E$, $\sum_{m=0}^{\infty}x_m\overline{e_m(t)}$ does not converge, and so the above computations without the limit and with $r=1$ show that the ordinary composition $S_{\overline{e}}\circ C$ does not exist.



By Lemma \ref{adjointlemma}, we obtain that $C\circledast A_{\overline{e}}=\ip{\cdot,\overline{\varphi}}_{\mu}\vec{x}$.
We then have that for any $\vec{y},\vec{z}\in \ell^2(\mathbb{N}_0)$,
\begin{align*}\ip{(C\circledast A_{\overline{e}})(S_{\overline{e}}\circledast C)\vec{y},\vec{z}}_{\ell^2}&=\ip{(C\circledast A_{\overline{e}})\ip{\vec{y},\vec{x}}_{\ell^2}\overline{\varphi},\vec{z}}_{\ell^2}\\
&=\ip{\ip{\vec{y},\vec{x}}_{\ell^2}\ip{\overline{\varphi},\overline{\varphi}}_{\mu}\vec{x},\vec{z}}_{\ell^2}\\
&=\ip{\ip{\vec{y},\vec{x}}_{\ell^2},\vec{z}}_{\ell^2}\\
&=\ip{C\vec{y},\vec{z}}_{\ell^2}.\end{align*}
Thus by Theorem \ref{Th:main}, $\mu\in\mathcal{M}(K_C)$.


\providecommand{\bysame}{\leavevmode\hbox to3em{\hrulefill}\thinspace}
\providecommand{\MR}{\relax\ifhmode\unskip\space\fi MR }
\providecommand{\MRhref}[2]{%
  \href{http://www.ams.org/mathscinet-getitem?mr=#1}{#2}
}
\providecommand{\href}[2]{#2}

\end{document}